\documentclass[12pt]{amsart}
\usepackage{RLintro}
\usepackage{enumitem}
\usepackage{tikz,tikz-cd}

\newcommand{\defn}[1]{\textit{#1}}
\newcommand{\prop}[1]{{\rm (#1)}}

\begin{document}

\title[Counterexample to a Gr\"obner approach for noetherianity]{A counterexample to a Gr\"obner approach for noetherianity of the twisted commutative algebra $\Sym(\Sym^2(\mathbf{C}^\infty))$}
\date{\today}

\author{Robert P. Laudone}
\address{Department of Mathematics, University of Michigan, Ann Arbor, MI}
\email{\href{mailto:laudone@umich.edu}{laudone@umich.edu}}
\urladdr{\url{https://robertplaudone.github.io/}}

\thanks{RL was supported by NSF grant DMS-2001992.}

\begin{abstract}
We resolve an open question posed by the authors of \cite{sym2noeth} in 2015 concerning a Gr\"obner theoretic approach to the noetherianity of the twisted commutative algebra $\Sym(\Sym^2(\bC^\infty))$. We provide a negative answer to their question by producing an explicit antichain. In doing so, we establish a connection to well studied posets of graphs under the subgraph and induced subgraph relation. We then analyze this connection to suggest future paths of investigation.
\end{abstract}

\maketitle

\tableofcontents

\section{Introduction}

\subsection{Statement of results.} Recall that a \defn{twisted commutative algebra} (tca) is a commutative $\bC$-algebra with an action of $\GL_\infty$ by algebra homorphisms for which it forms a polynomial representation. In \cite{sym2noeth}, the authors prove that the twisted commutative algebra $\Sym(\Sym^2(\bC^\infty))$ is noetherian in characteristic $0$. They then propose a different method of proof and ultimately pose the question of whether a partially ordered set $(\cM,\sqsubseteq)$ is noetherian. The noetherianity of this poset would imply the noetherianity of the twisted commutative algebra $\Sym(\Sym^2(\bC^\infty))$ in any characteristic. This question has been open since 2015, the main result of this paper is providing a negative answer by constructing an infinite antichain in the poset:

\begin{theorem} \label{thm: A}
The poset $(\cM,\sqsubseteq)$ described in \cite[Question 5.2]{sym2noeth} (and \S \ref{sec: sym2grob}) is not noetherian.
\end{theorem}

To construct this counterexample, we establish a connection to graph theory that, to our knowledge, has not been seen before in investigating noetherianity results about twisted commutative algebras. Up to this point, all of the noetherianity results in this vein have relied on some variant of Higman's lemma, which one can view as a ``one dimensional'' result in that it is concerned with words. The use of graph theory can be seen as an application of ``higher dimensional'' combinatorics.  We believe such a connection will be necessary if one wishes to use Gr\"obner and combinatorial methods to approach noetherianity of higher degree twisted commutative algebras.

\subsection{Motivation.} Recently, researchers have discovered many large algebraic structures that have surprising finiteness properties up to natural symmetries. Examples include ${\bf FI}$ \cite{fimodule}, $\cH$ \cite{heckeCat}, as well as the collection of Veronese \cite{Veronese} and Pl\"ucker ideals \cite{Plucker}. Twisted commutative algebras are another class of examples, but are still largely not understood; see \S \ref{subsec: tca} for the general definition. In the setting of these algebraic structures, we often consider sequences of modules $M_n$ that are ``compatible'' in a certain sense and the finiteness properties we seek are some sort of stabilization as $n$ gets large. In all of these cases, one of the most important finiteness properties is noetherianity. For a tca $A$, there is a notion of a finitely generated $A$-module, and $A$ is said to be \defn{noetherian} if any submodule of a finitely generated $A$-module is also finitely generated.

All degree $1$ tca's are easily seen to be noetherian, for more information on tca's and the proof of this fact we refer the reader to \cite{introtca}. In fact, modules over the tca $\Sym(\bC^\infty)$ are equivalent to the ${\bf FI}$-modules of \cite{fimodule} under Schur-Weyl duality. As soon as one starts to consider tca's generated in degree larger than one, much less is known. Indeed, only six degree two tca's are known to be noetherian, see \cite{sym2noeth,periplectic,spnoeth} for details.

All of these results stem from a similar idea. Namely, one studies the torsion elements in the category of modules for the tca as well as the generic category, which is the Serre quotient by the torsion subcategory. One then investigates how both of these pieces glue together to deduce noetherianity. Although the idea is similar in all cases, the execution is often specific to the example, involved and characteristic dependent. 

Draisma was able to show that all tca's finitely generated in any degree are topologically noetherian, i.e. radicals satisfy the ascending chain condition \cite{draisma}. One could hope that a similar result holds algebraically. This is actually one of the major open problems in the theory of tca's. As of now, we seem far from proving something this strong, and to get there we may need to seek other methods of proof that are more easily generalizable.

In \cite{sym2noeth}, the authors suggest a potential step in that direction, namely trying to apply Gr\"obner methods for proving noetherianity. This method already has the benefit of being independent of characteristic. These Gr\"obner techniques are successfully applied in \cite{Veronese,Plucker,heckeCat,grob}.

Such an approach also works for simple examples of tca's, for example $\Sym(\bC^{\infty} \oplus \bC^{\infty})$, and ultimately boils down to an application of Higman's lemma, but in degree $2$ more complications arise. We include the details of the $\Sym(\Sym^2(\bC^{\infty}))$ case in \S \ref{sec: sym2grob} but refer the reader to \cite[Section 5]{sym2noeth} for details about the degree $1$ case. After outlining this more combinatorial approach to noetherianity of $\Sym(\Sym^2(\bC^\infty))$, the authors in \cite{sym2noeth} end with the question of whether a poset they construct is noetherian, which would ultimately imply noetherianity of $\Sym(\Sym^2(\bC^\infty))$. In answering this question, even negatively, we hope to provide potential paths for further research in this direction, as well as motivation to revisit this Gr\"obner approach.

\subsection{Idea behind the proof.} The proof of Theorem \ref{thm: A} relies on connecting the poset $(\cM,\sqsubseteq)$ of matchings ordered under certain allowable moves to the poset of permutation graphs ordered by what these allowable moves induce on the underlying permutations.

To do this, we first restrict to a subset of all perfect matchings which we can connect to words. We then interpret the partial order on perfect matchings in terms of their word counterparts  (Propositions \ref{prop: TypeIequiv}, \ref{prop: TypeIIequiv}). We can further view these words as word representations of permutations (Proposition \ref{prop: matchingConnection}). We then consider the corresponding permutation graphs. We label a well-known antichain of graphs for the subgraph relation, proving these graphs are permutation graphs and associating to each a permutation. We then argue that these permutations provide an antichain for $(\cM,\sqsubseteq)$ by studying what happens to the graphs as we change the corresponding permutations. This allows us to construct an infinite antichain in our original poset.

\subsection{Outline.} In \S \ref{sec: background} we provide all the relevant background material on tca's, representations of $\GL$, and graph theory. In \S \ref{sec: sym2grob} we recall the setup from \cite[\S 5.3]{sym2noeth} for the question we answer. In \S \ref{sec: settingUp}, we describe all the necessary setup for the counterexample. This is the section where we establish the connection between the poset of perfect matchings ordered by certain allowable moves and the poset of permutations and their corresponding permutation graphs. In \S \ref{sec: counterEx} we use the connection in the previous section to construct an explicit counterexample, showing the poset $(\cM,\sqsubseteq)$ is not noetherian. Finally, in \S \ref{sec: goingForward} we outline future directions of research we are actively investigating that could stem from the techniques established in this paper to prove noetherianity of higher degree tca's via Gr\"obner methods.

\subsection{Acknowledgments.} We thank Rohit Nagpal, Steven Sam and Andrew Snowden for posing this question. We especially thank Steven and Andrew for many helpful conversations about this topic.

\section{Background} \label{sec: background}

\subsection{Important Definitions.}\label{subsec: tca} By $\GL_\infty$, we mean $\bigcup_{n \geq 1} \GL_n$. A representation of $\GL_\infty$ is \defn{polynomial} if it is a subquotient of a possibly infinite direct sum of representations of the form $(\bC^\infty)^{\otimes k}$. Polynomial representations of $\GL_\infty$ are semi-simple and all the simple modules are indexed by partitions. That is, the simple modules are precisely $\bS_\lambda(\bC^\infty)$, where $\bS_\lambda$ is the Schur functor associated to the partition $\lambda$. A polynomial representation is said to be \defn{finite length} if it is a direct sum of finitely many simple representations. We refer the reader to \cite{introtca} for details.

A \defn{twisted commutative algebra} (tca) is a commutative unital $\bC$-algebra $A$ equipped with an action of $\GL_\infty$ by $\bC$-algebra homomorphisms such that $A$ forms a polynomial representation of $\GL_\infty$.

\subsection{Admissible weights.}\label{subsec: admissibleWeights} A weight of $\GL_\infty$ is a sequence of non-negative integers $w = (w_1,w_2,\dots)$ such that $w_i = 0$ for $i \gg 0$. The classical results about weight space decomposition of polynomial representations for $\GL_n$ carry over to the infinite setting. Namely, if $V$ is any polynomial representation of $\GL_\infty$ then we have $V = \bigoplus_{w} V_w$ where $V_w$ is the weight space of weight $w$. A weight $w$ is \defn{admissible} if all the $w_i$ are either $1$ or $0$. An \defn{admissible weight vector} is an element of $V_w$ where $w$ is an admissible weight. We will make use of the following fact: if $V$ is a polynomial representation of $\GL_\infty$ then $V$ is generated, as a representation, by its admissible weight vectors.

\subsection{Permutation Graphs.}\label{sec: permGraphs} We assume the reader has a basic background in graph theory and combinatorics. For a permutation $\sigma$ of $n$, we define the \defn{permutation graph} $G_\sigma$ to have vertex set $V = [n]$ and edge set $E(G) = \{(i,j) \; | \; i < j, \sigma(i) > \sigma(j)\}$. A permutation graph is not a directed graph, but following the convention of Kho and Ree we write an edge $(i,j) \in E$ with $i < j$, i.e. as an ordered pair instead of writing $\{i,j\} \in E$. We call such a pair $(i,j)$ with $i < j$ but $\sigma(i) > \sigma(j)$ an \defn{inversion} in $\sigma$. When we focus on a single element $i$, we say that another element $j$ is an inversion with $i$ if $(i,j)$ is an inversion.

Not every graph is a permutation graph. Koh and Ree in \cite[Theorem 3.2]{permGraph} showed that permutation graphs are completely characterized by the following properties:
\begin{itemize}
\item[\prop{P1}\label{P1}] $E$ is transitive, i.e., if $(i,j) \in E$ and $(j,k) \in E$, then $(i,k) \in E$.
\item[\prop{P2}\label{P2}] If $(i,k) \in E$ and $i < j < k$ for some $j$, then it must hold that $(i,j) \in E$ of $(j,k) \in E$.
\end{itemize}

This characterization allows one to show a graph is a permutation graph by constructing an appropriate labeling of its vertices satisfying \prop{P1} and \prop{P2}.

\subsection{Well-Quasi-Ordering.} Let $(\cP,\leq)$ be a partially ordered set (poset). When discussing a poset we will often supress the partial order and just write $\cP$. An \defn{antichain} in $\cP$ is a (potentially infinite) sequence of elements of $\cP$, $p_1,p_2,p_3,\dots$, such that $p_i \nleq p_j$ for any $j > i$. We say that $\cP$ is \defn{well-quasi-ordered}, also referred to as \defn{noetherian}, if $\cP$ is well founded and does not have an infinite antichain with respect to $\leq$. Equivalently, $\cP$ is noetherian if any infinite sequence of elements $p_1,p_2,\dots$ in $\cP$ contains some increasing pair $p_i \leq p_j$ with $i < j$.

When proving a poset is noetherian, one often proves that any infinite sequence has two elements that are comparable. When disproving noetherianity, one constructs an infinite antichain.

\section{Gr\"obner approach to noetherianity}\label{sec: sym2grob} In \cite[\S 5.3]{sym2noeth}, the authors propose a Gr\"obner theoretic approach to proving noetherianity of the tca $\Sym(\Sym^2(\bC^\infty))$. This paper is concerned with providing a negative answer to a question they pose after setting up this approach, so we will include the setup. Let $A = \Sym(\Sym^2(\bC^\infty))$. Let $x_{i,j}$, with $i \le j$, be a basis for $\Sym^2(\bC^\infty)$, so that $A = \bC[x_{i,j}]$.

Let $\cM$ be the set of undirected matchings $\Gamma$ on $\bN$. Given $\Gamma,\Gamma' \in \cM$, we define $\Gamma \to \Gamma'$ if one of the following two conditions hold,
\begin{itemize}
\item $\Gamma'$ is obtained from $\Gamma$ by adding a single edge.
\item There exists an edge $(i,j)$ in $\Gamma$ such that $j+1$ is not in $\Gamma$, and $\Gamma'$ is obtained from $\Gamma$ by replacing $(i,j)$ with $(i,j+1)$. (Here $i< j$ or $j < i$).
\end{itemize}
The authors in \cite{sym2noeth} call $\Gamma \to \Gamma'$ a \defn{Type I move}. We refer to the first bullet point as a Type I(a) move and the second as a Type I(b) move. They define $\Gamma \leq \Gamma'$ if there is a sequence of type I moves transforming $\Gamma$ to $\Gamma'$. This partially orders $\cM$. On the level of graphs, Type I moves allow you to add edges connecting valence $0$ vertices and to shift existing edges up by one vertex if the next vertex is empty.

They then define a total order $\preceq$ on $\cM$. First, suppose that $i < j$ and $k < \ell$ are elements of $\bN$. Define $(i,j) \preceq (k,\ell)$ if $j < \ell$, or $j = \ell$ and $i \leq k$. They then expand this definition to a lexicographic order on $\cM$. Explicitly, let $\Gamma$ and $\Gamma'$ be two elements of $\cM$ with $e_1 \preceq e_2 \preceq \cdots \preceq e_n$ and $e_1' \preceq e_2' \preceq \cdots \preceq e_m'$ their edges listed in increasing order. Then $\Gamma \preceq \Gamma'$ if $n < m$, or if $n = m$ and $(e_1,\dots,e_n) \preceq (e_1',\dots,e_m')$ under the lexicographic order reading from right to left, to stay consistent with the definition on single edges.

Given $\Gamma \in \cM$, we define $m_\Gamma = \prod_{(i,j) \in \Gamma} x_{i,j}$. Every admissible weight vector is a sum of $m_\Gamma$'s, and every polynomial representation of $\GL_\infty$ is generated by its admissible weight vectors (\S \ref{subsec: admissibleWeights}), so we can restrict our attention to these elements. Using $\preceq$, we let the \defn{initial element} of any $f \in A$ be the largest $\Gamma$ under $\preceq$ such that $m_\Gamma$ appears with nonzero coefficient in $f$. We denote the initial variable by ${\rm in}(f)$.

For any ideal $I$ of $A$, let ${\rm in}(I) = \{{\rm in}(f) \; | \; f \in I\}$ be the set of initial elements in $I$. In \cite{sym2noeth}, the authors observe that ${\rm in}(I)$ is closed under Type I moves, and therefore forms a poset ideal of the poset $(\cM, \leq)$. But this poset is not noetherian. This leads to the introduction of more ``types'' of moves to hopefully remedy this situation. All of these moves come from allowing $\GL_\infty$ to act in a way that respects the total order $\preceq$ and therefore preserves the initial ideal ${\rm in}(I)$. Each new type of move is finding a slightly more complex action.

The next type of moves the authors define as follows. We include pictures illustrating the moves and refer the reader to \cite{sym2noeth} for the explicit definition. We do this because the pictures are generally a much clearer illustration of the moves and it is not hard to translate between the perfect matchings and monomials.
\begin{equation} \label{TypeIImoves}
\begin{tikzcd}[sep=small]
a \ar[dash,rrr,bend left = 40] \ar[dotted,dash,r] &b \ar[dash,bend left = 40, r]&c&d
\end{tikzcd}
\implies
\begin{tikzcd}[sep=small]
a \ar[dash,rr,bend left = 45]  &b\ar[dotted,dash,r] \ar[dash,bend left = 45, rr]&c&d
\end{tikzcd}
\implies
\begin{tikzcd}[sep=small]
a \ar[dash,r,bend left = 45]  &b &c\ar[dash,bend left = 45, r]&d
\end{tikzcd},
\end{equation}
where $a < b < c < d$ and the dotted lines indicate that any element there is either not an edge or is connected to a number larger than $c$. Write $\Gamma \implies \Gamma'$ to indicate that $\Gamma'$ is related to $\Gamma$ by a sequence of any of the two modifications in \eqref{TypeIImoves}. These are called ``Type II'' moves. We refer to the first move as a Type II(a) move and the second as Type II(b). One can then place a new partial order $\sqsubseteq$ on $\cM$ where $\Gamma \sqsubseteq \Gamma'$ if there exists a sequence of moves (of any type) taking $\Gamma$ to $\Gamma'$. The authors observe that these moves respect the initial ideal so that ${\rm in}(I)$ is still a poset ideal of $(\cM, \sqsubseteq)$. They pose the following question:

\begin{question} \label{question: noethq}
Is the poset $(\cM,\sqsubseteq)$ noetherian?
\end{question}

This question has been open since 2015. The remainder of this paper is dedicated to answering this question in the negative. We produce an explicit counterexample to the noetherianity of this poset. In doing so, we establish a connection between this poset and a poset of graphs.

\section{Setting up the counterexample} \label{sec: settingUp}
We begin by restricting ourselves to a particular subset of perfect matchings where the first vertices $\{1,\dots,n\}$ are all paired with vertices in $\{n+1,\dots,2n\}$. We call such perfect matchings \defn{intertwined}. We will often use visual representations of these matchings in terms of graphs. When we do so, we are always assuming the graphs have vertices labeled $\{1,\dots,2n\}$ placed in ascending order form left to right. An important property of intertwined perfect matchings is that they do not have a subgraph of the following type:
\[
\begin{tikzcd}
a \ar[bend left=30,dash,r] & b & c \ar[bend left=30,dash,r]& d,
\end{tikzcd}
\]
with $a < b < c < d$. One reason for our restriction to this class is we will never make use of the Type II(b) move in \cite{sym2noeth} because this would create a non-intertwined perfect matching and once a matching is non-intertwined none of the moves can make it intertwined again. These perfect matchings on $2n$ letters are also easily encoded by words on the alphabet $[n]$, all with distinct letters. Indeed, given an intertwined perfect matching with edges $(e_1,n+1),\dots,(e_n,2n)$ this corresponds bijectively to the word $e_{n}\cdots e_1$. The opposite direction is clear. 

Intuitively, to read off the corresponding word from an intertwined perfect matching you work from right to left and write down the number of the origin vertex connected to each terminal vertex in your matching. For an intertwined perfect matching $\Gamma$, we denote its corresponding word by $w_\Gamma$.

\begin{example}
To illustrate this bijection consider the following intertwined perfect matching on six vertices,
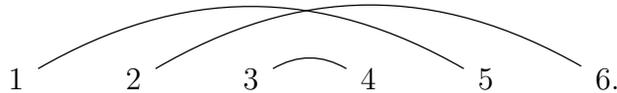
\[
\begin{tikzcd}
1 \ar[bend left = 30,dash,rrrr] &2 \ar[bend left = 30,dash,rrrr] & 3  \ar[bend left = 30,dash,r]  & 4 & 5 & 6.
\end{tikzcd}
\]
This corresponds to the word $213$. It is also not hard to go in the other direction. For example, the word $312$ corresponds to
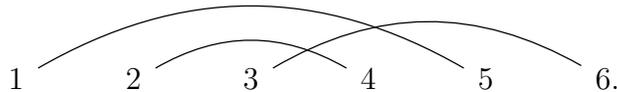
\[
\begin{tikzcd}
1 \ar[bend left = 30,dash,rrrr] &2  \ar[bend left = 30,dash,rr] & 3  \ar[bend left = 30,dash,rrr]  & 4 & 5 & 6.
\end{tikzcd}
\]
\end{example}

Now we wish to understand Type I and II moves in terms of the words corresponding to intertwined perfect matchings. First, we must make a definition. For a word $w = w_1\cdots w_n$ with distinct letters in $[m]$ with $m \ge n$, we let the \defn{reduced word} of $w$ denoted ${\rm red}(w)$ be the word where we replace the letters $w_{i_1} < w_{i_2} < \cdots < w_{i_n}$ with $1 < 2 < \cdots < n$. For example, the reduced word of $364$ is $132$. Type I moves correspond to order preserving injections and adding additional letters to the word. More explicitly, we say one word $w_1\cdots w_n$ is \defn{order isomorphic} to a subword of another word $s_1 \cdots s_m$ with $m \geq n$ if there exists $s_{i_1}\cdots s_{i_n}$, $s_{i_j} \geq w_{j}$ for $1 \leq j \leq n$, with ${\rm red}(s_{i_1}\cdots s_{i_n})= {\rm red}(w_1\cdots w_n)$. Then we have

\begin{proposition} \label{prop: TypeIequiv}
An intertwined perfect matching $\Gamma$ can be transformed into another intertwined perfect matching $\Gamma'$ via Type I moves if and only if $w_\Gamma$ is order isomorphic to a subword of $w_{\Gamma'}$.
\end{proposition}

\begin{proof}
Suppose first that we have $\Gamma \to \Gamma'$. Let $(a_n,n+1) \preceq (a_{n-1},n+2) \preceq \cdots \preceq (a_2,2n-1) \preceq (a_1,2n)$ be the edges of $\Gamma$ listed in the lex order described in $\S$\ref{sec: sym2grob}. We can write the edges in this way because the matching is perfect and intertwined. Notice that we obtain the corresponding word $w_\Gamma$ as $a_1a_2\cdots a_n$.

Consider each of these edges in $\Gamma$ and where they are sent in $\Gamma'$ after applying the Type I moves. We note that after each Type I move, the matching may no longer be a perfect matching, but the final result, i.e. $\Gamma'$ will be. Importantly, though, the matching will always remain intertwined. Any Type I (a) move performed on one of the edges from $\Gamma$ sends $(i,j)$ to $(i,j+1)$ or $(i+1,j)$ but only if $j+1$ or $i+1$ respectively is valence zero. This does not change the order of the edges in $\Gamma$, i.e. the listed edges do not swap in the lex order. Any Type I(b) move adds a single edge to $\Gamma$, but again does not change the order of the edges in the original matching $\Gamma$. Since the order of the edges in $\Gamma$ was not changed, if we consider the subword corresponding to the image of these edges in $w_{\Gamma'}$, we recover a word that is order isomorphic to $w_\Gamma$.
%

Conversely, suppose we have $w_\Gamma$ order isomorphic to a subword of $w_{\Gamma'}$ for some intertwined perfect matchings $\Gamma$ and $\Gamma'$. This means there is an order preserving injection of the letters of $w_\Gamma$ into the letters of $w_{\Gamma'}$. This corresponds to shifting edges up, i.e. Type I(b) moves. We then fill in the remaining letters of $w_{\Gamma'}$ using Type I(a) moves.
\end{proof}

Type II(a) moves are a bit more subtle. 

\begin{proposition} \label{prop: TypeIIequiv}
Applying a Type II(a) move to an intertwined perfect matching $\Gamma$ corresponds to swapping two letters $i < j$ if $i$ appears before $j$ and all the numbers between $i$ and $j$ appear before $j$ when reading from left to right.
\end{proposition}

\begin{proof}
It is clear from translating the definition of a Type II(a) move to the word representation of a perfect matching that we are allowed to apply a Type II(a) move if and only if the corresponding letters we wish to swap are $i < j$ with $i$ appearing before $j$. Furthermore, the restriction that any element between the vertices labeled with $a$ and $b$ in \eqref{TypeIImoves} must be connected to a vertex larger than $c$ means that every number strictly between $i$ and $j$ must appear before $j$.
\end{proof}

Words corresponding to intertwined perfect matchings on $2n$ vertices can be thought of as permutations of $[n]$ in one-line notation, this is also sometimes called the \defn{word representation} of a permutation. We note that we think of this connection on the level of posets. We will adopt this viewpoint because well-quasi-orders on permutations have received a good amount of attention since the early 2000s, and we would like to use techniques and results from this area.

Type I moves in this setting then correspond to the well studied pattern containment order, which is also known to not be well-quasi-ordered. Numerous examples exist to demonstrate this which arise in various settings, to name a few: Laver \cite{wqoFinSeq}, Pratt \cite{compPerm} and Speilman and B\'ona \cite{infAntiPerm}. For a straight-forward antichain, we particularly recommend Speilman and B\'ona's paper. In this way, we find many other counterexamples to perfect matchings with just Type I moves being well-quasi-ordered. 

The addition of Type II moves and the partial order it induces on permutations has, to our knowledge, not received any attention in the literature; especially in the context of combining Type I and Type II moves to compare permutations of any length. We use $\leq_{i}$ to denote this partial order on permutations and $p_\Gamma$ to denote the permutation corresponding to $\Gamma$. Propositions \ref{prop: TypeIequiv} and \ref{prop: TypeIIequiv} imply,

\begin{proposition} \label{prop: matchingConnection}
For two intertwined perfect matchings $\Gamma,\Gamma'$, $\Gamma \sqsubseteq \Gamma'$ if and only if $p_\Gamma \leq_i p_{\Gamma'}$.
\end{proposition}

\begin{remark}
This partial order is closely related to the Bruhat order, indeed it is strictly weaker than the Bruhat order when we restrict to permutations of a specific size. The (strong) Bruhat order allows you to swap $i < j$ with $i$ appearing before $j$ if all the numbers between $i$ and $j$ appear before $i$ or after $j$.
\end{remark}

When constructing infinite antichains for permutation classes, it is often convenient to work instead with the corresponding permutation graph. When working with permutation graphs, we will consider the graphs as having vertices labeled by $\{1,\dots, n\}$, but when comparing these graphs under the induced subgraph or subgraph relation, the labelings are irrelevant. We are only concerned with the graph itself. It is well known, and not hard to show,

\begin{proposition} \label{prop: TypeIgraph}
If $\sigma$ is order isomorphic to a sub permutation of $\tau$ then $G_\sigma$ is an induced subgraph of $G_\tau$
\end{proposition}

\begin{proof}
See for example \cite[\S 1]{wqoPermGraph}. 
\end{proof}

The converse is not true because the map from permutations to their graphs is many-to-one. Indeed, a permutation and its inverse have the same permutation graph and are not order isomorphic.

This correspondence, though, is used to either construct counterexamples on the graph theoretic side that carry over to counterexamples of permutations, or to prove that a class of permutation graphs is well-quasi-ordered by showing that the corresponding class of permutations is well-quasi-ordered \cite{wqoPermGraph}.

Now we are ready to see a few properties that Type II(a) moves have in the graph theoretic picture,

\begin{proposition} \label{prop: moreEdges}
Applying a Type II move to $\sigma$ always increases the number of edges in the corresponding permutation graph, while keeping the number of vertices the same.
\end{proposition}

\begin{proof}
We break this proof into cases depending the positioning of certain elements of the permutation. For this purpose, suppose we are going to swap the letters $a$ and $b$ in $\sigma$, with $a < b$ and $\sigma(a) < \sigma(b)$. We break the remaining letters into the following categories
\begin{enumerate}[label = (\Alph*)]
\item All elements less than $a$,
\item All elements between $a$ and $b$,
\item All elements larger than $b$.
\end{enumerate}
And we section off the places these elements could appear in $\sigma$ in the following way
\begin{enumerate}[label = (\Roman*)]
\item Appearing before $a$,
\item Appearing between $a$ and $b$,
\item Appearing after $b$.
\end{enumerate}
Notice that based on the restriction of when we can apply a Type II move, we can never have elements in (B) appearing in (III). We examine what happens to the graph for each possible pairing when we swap $a$ and $b$. This accounts for all the possible changes to the graph as every element of $\sigma$ falls into some pairing of these categories. We first consider all the elements in (A), i.e. those less than $a$.
\begin{enumerate}[label = --,leftmargin = *]
\item (A) and (I): Before the swap, these vertices were not connected to $a$ or $b$, after the swap this stays the same so there are no additional edges.
\item (A) and (II): Before the swap, these vertices were connected to $a$ but not $b$. After the swap, we remove all the edges from these vertices to $a$ and add edges from these vertices to $b$.
\item (A) and (III): We do not have to change anything because all of these vertices are connected to both $a$ and $b$ before and after the swap.
\end{enumerate}
We similarly consider all the other possible vertex pairings
\begin{enumerate}[label = --,leftmargin = *]
\item (B) and (I): No additional edges. All of these vertices are connected to $a$ and not $b$. They stay this way after the swap.
\item (B) and (II): Before the swap, these vertices were not connected to $a$ or $b$. After the swap, we must add edges from these vertices to both $a$ and $b$.
\item (B) and (III): Not allowed.
\end{enumerate}
Finally,
\begin{enumerate}[label = --,leftmargin = *]
\item (C) and (I): Before the swap there were edges from these vertices to both $a$ and $b$, this stays the same after the swap.
\item (C) and (II): Before the swap there were edges from these vertices to $b$ but not $a$. After the swap, we must remove all the edges to $b$ and add edges from each of these vertices to $a$ instead.
\item (C) and (III): Before the swap these vertices were not connected to $a$ or $b$, after the swap this stays the same.
\end{enumerate}
In all of these cases the number of edges either stays the same, or increases. But when we swap $a$ and $b$ in $\sigma$, we also have to add an edge between $a$ and $b$. As a result, the number of edges in the graph that results from applying a Type II move to $\sigma$ has strictly more edges than $G_\sigma$.
\end{proof}

\begin{remark}
Another way to see that the number of edges increases is to notice that Type II moves on permutations is a weaker version of the Bruhat order. That is $\sigma \leq_i \tau$ implies $\sigma \leq_B \tau$. We do not prove this here because we prefer to include more details about our specific case, but the previous result is an immediate corollary of this. Indeed, for two permutations of the same size, $\sigma \leq_B \tau$ means $\ell(\sigma) < \ell(\tau)$, which is equivalent to saying there are more edges in the permutation graph.
\end{remark}

\begin{example}
We include examples of how this proposition works.
%
Consider $2143 \leq_i 3142$. We had originally hoped that Type II and Type I moves would imply the subgraph relation on permutation graphs, but this turned out to not be the case. We will, however, use a well known antichain for the subgraph relation on graphs to produce the antichain for $(\cM,\sqsubseteq)$. These permutations correspond to the following graphs

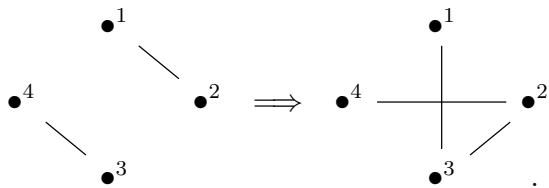
\[
\begin{tikzcd}[sep = small]
&\bullet^1\ar[dash,dr]& \\
\bullet^4\ar[dash,dr]&&\bullet^2\\
&\bullet^3& \\
\end{tikzcd}
\implies
\begin{tikzcd}[sep = small]
&\bullet^1\ar[dash,dd]&\\
\bullet^4\ar[dash,rr]&&\bullet^2\ar[dash,dl]\\
&\bullet^3&.\\
\end{tikzcd}
\]
We include the labels to illustrate how one constructs permutation graphs. It is not hard to see here that the graph on the right has more edges than the graph on the left. Furthermore, we can obtain the graph on the right by following the procedure outlined in Proposition \ref{prop: moreEdges}. $1$ is in category (A) and (II), so we must remove the edge $(1,2)$ and add the edge $(1,3)$. $4$ is in category (C) and (II), so we must remove the edge $(3,4)$ and add the edge $(2,4)$. Finally, we always add the edge $(2,3)$ and we obtain the new graph.

Now let us look at an example comparing two permutations of different sizes. Consider $2143 \leq_i 34152$. We can realize this relation by $2143 \implies 3142 \to 34152$ where we first apply a Type II move to swap $2$ and $3$, then apply Type I moves sending $1 \to 1$, $2 \to 2$, $3 \to 3$ and $4 \to 5$ and adding $4$ into the second position. This corresponds to the following picture on graphs,
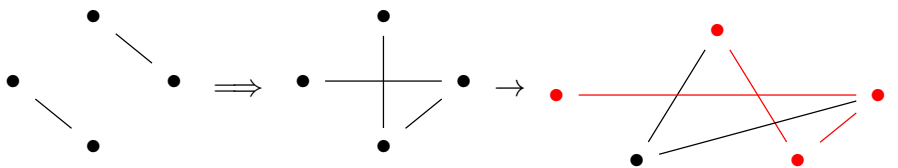
\[
\begin{tikzcd}[sep = small]
&\bullet\ar[dash,dr]& \\
\bullet\ar[dash,dr]&&\bullet\\
&\bullet& \\
\end{tikzcd}
\implies
\begin{tikzcd}[sep = small]
&\bullet\ar[dash,dd]&\\
\bullet\ar[dash,rr]&&\bullet\ar[dash,dl]\\
&\bullet&\\
\end{tikzcd}
\to
\begin{tikzcd}[sep=small]
&&{\color{red} \bullet}\ar[dash,ddl]\ar[dash,ddr,color=red]&&\\
{\color{red} \bullet}\ar[dash,rrrr,color=red]&&&&{\color{red} \bullet}\ar[dash,dl,color=red]\ar[dash,dlll]\\
&\bullet&&{\color{red} \bullet}&
\end{tikzcd}.
\]
It is not hard to see that the second graph is an induced subgraph of the final graph, we colored it red for clarity. Notice, also, that the first graph is not an induced subgraph of the last. This is because we need more than Type I moves to realize the connection between their corresponding perfect matchings.
\end{example}

\begin{proposition} \label{prop: preserveCycles}
Type I and Type II moves preserve cycles in permutation graphs. That is, if $G$ is a permutation graph with a cycle, after any application of Type I and II moves to any permutation associated to $G$, the resulting graph will still contain a cycle.
\end{proposition}

\begin{proof}
Let $w = w_1 w_2 \cdots w_m$ be any permutation associated to $G$. Any Type I move will maintain a cycle by Proposition \ref{prop: TypeIgraph}, so it remains to argue that Type II moves also preserve cycles. Pick any cycle in the graph. Suppose $w_{i_1}w_{i_2}\cdots w_{i_n}$ with $1 \leq i_1 < i_2 < \cdots < i_n \leq m$ is a minimal subpermutation whose induced subgraph contains the cycle, and that the cycle is given by $(w_{j_1},w_{j_2}, \dots, w_{j_{n}}, w_{j_1})$ where $\{j_1,j_2,\dots,j_n\} = \{i_1,i_2,\dots,i_n\}$ as unordered sets. Here we are walking along the cycle and reading off the labelings on the vertices. If we apply any Type II move that does not involve these elements, the cycle is clearly still present. Now there are a few cases to consider.

The first case is if we apply a Type II move to some $w_i$ and $w_j$ both present in the cycle, with $w_i < w_j$ and $i < j$. If we are allowed to swap $w_i$ and $w_j$ they could not appear consecutively in the cycle because they do not have an edge between them. Suppose we have $(w_i, w_{\alpha_1}, \cdots, w_{\alpha_m}, w_j)$ as the path between $w_i$ and $w_j$. We may assume that neither $w_i$, nor $w_j$ are at the beginning of the cycle because we can start our cycle from anywhere. There are six subcases to consider:

\begin{enumerate}[leftmargin = *,label=\roman*)]
\item If $m > 1$ and $w_i < w_{\alpha_1}$ and $w_j < w_{\alpha_m}$ we have a new cycle $(w_i, w_{\alpha_1}, \dots, w_{\alpha_m}, w_i)$.
\item If $m = 1$ and $w_i < w_{\alpha_1}$ and $w_j < w_{\alpha_1}$ we have a new cycle $(w_i, w_{\alpha_1}, w_j, w_i)$.
\item If $m > 1$ and $w_i > w_{\alpha_1}$ and $w_j > w_{\alpha_m}$ we have a new cycle $(w_j, w_{\alpha_1}, \dots, w_{\alpha_m}, w_j)$.
\item If $m = 1$ and $w_i > w_{\alpha_1}$ and $w_j > w_{\alpha_1}$ we have a new cycle $(w_j, w_i, w_{\alpha_1}, w_j)$.
\item For any $m \geq 1$, if $w_i < w_{\alpha_1}$ and $w_j > w_{\alpha_m}$ we have a new cycle $(w_i, w_{\alpha_1}, \dots, w_{\alpha_m}, w_j, w_i)$.
\item For any $m \geq 1$, if $w_i > w_{\alpha_1}$ and $w_j < w_{\alpha_m}$ we have a new cycle $(w_j, w_{\alpha_1}, \dots, w_{\alpha_m}, w_i, w_j)$.
\end{enumerate}

The second case is if we apply a Type II move to some $w_i$ present in the cycle and some other element $w_j$ not in the cycle, we either have $w_i < w_j$ with $i < j$ or $w_j < w_i$ with $j < i$. Both cases are similar, so we only discuss the first one. Suppose $\cdots w_{\alpha_1} w_i w_{\alpha_2} \cdots$ is the part of the cycle where $w_i$ appears. We may assume without loss of generality that $w_i$ is not the beginning of our cycle because we can start the cycle from anywhere. Once again there are four subcases the consider:

\begin{enumerate}[leftmargin = *,label=\roman*)]
\item If $w_i > w_{\alpha_1}$ and $w_i > w_{\alpha_2}$, then we can replace $(w_{\alpha_1}, w_i, w_{\alpha_2})$ in the cycle with $(w_{\alpha_1}, w_j, w_{\alpha_2})$.
\item If $w_i > w_{\alpha_1}$ and $w_i < w_{\alpha_2}$, then we can replace $(w_{\alpha_1}, w_i, w_{\alpha_2})$ in the cycle with $(w_{\alpha_1}, w_j, w_i, w_{\alpha_2})$.
\item If $w_i < w_{\alpha_1}$ and $w_i > w_{\alpha_2}$, then we can replace $(w_{\alpha_1}, w_i, w_{\alpha_2})$ in the cycle with $(w_{\alpha_1}, w_i, w_j, w_{\alpha_2})$.
\item If $w_i < w_{\alpha_1}$ and $w_i < w_{\alpha_2}$, then we can just keep $(w_{\alpha_1}, w_i, w_{\alpha_2})$ in the cycle.
\end{enumerate}

This handles all possible cases and shows that whenever we apply a Type I or Type II move, if a graph contains a cycle, it will still contain one after the move.
\end{proof}

\begin{example}
Consider the permutation $3214$. This has the following permutation graph
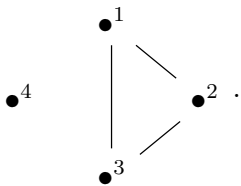
\[
\begin{tikzcd}[sep=small]
&\bullet^1\ar[dash,dd]\ar[dash,dr]&\\
\bullet^4&&\bullet^2\ar[dash,dl]\\
&\bullet^3&
\end{tikzcd}.
\]
We can write the cycle as $(3,2,1,3)$. Say we wish to apply the Type II move swapping $2$ and $4$. We are taking $w_2$ and swapping it with $w_4$ which is not in the cycle. As a result, we fall into the second case of Proposition \ref{prop: preserveCycles}. The part of the cycle we are concerned with directly to the left and right of $2$ is $(3,2,1)$. We see that $2 < 3$ but $2 > 1$, so we fall into sub-case iii). The proof of the Proposition tells us that the resulting permutation $3412$ will contain the cycle $(3,2,4,1,3)$, which is indeed the case
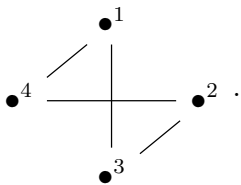
\[
\begin{tikzcd}[sep=small]
&\bullet^1\ar[dash,dd]\ar[dash,dl]&\\
\bullet^4\ar[dash,rr]&&\bullet^2\ar[dash,dl]\\
&\bullet^3&
\end{tikzcd}.
\]
Now suppose we wish to apply another Type II move, swapping $1$ and $2$. Both of these labels appear in the cycle, so we fall into the first case. The path between these two entries in the cycle is $(2,4,1)$. In this case we have $m=1$ and $w_{\alpha_1} = 4$. We see that $2 < 4$ and $1 < 4$ so we fall into subcase ii). The proof of the Proposition tells us that we have the cycle $(1,4,2,1)$ in our new graph which we see is the case
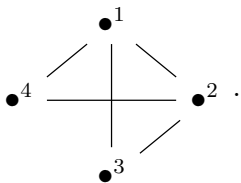
\[
\begin{tikzcd}[sep=small]
&\bullet^1\ar[dash,dd]\ar[dash,dl]\ar[dash,dr]&\\
\bullet^4\ar[dash,rr]&&\bullet^2\ar[dash,dl]\\
&\bullet^3&
\end{tikzcd}.
\]
\end{example}

We mention one other fact about how these moves affect permutation graphs because we implicitly use it, so thought it worth explicitly mentioning:

\begin{proposition} \label{prop: connected}
If two vertices in a permutation graph were connected before an application of a Type I or II move to the underlying permutation, they remain connected after. In particular, neither type of move can disconnect a connected component.
\end{proposition}

\begin{proof}
For Type I moves, this follows immediately from Proposition \ref{prop: TypeIgraph}. 

For Type II moves we just have to notice that if $w_i$ and $w_j$ are adjacent to each other, there is still a path between them after an application of Type II moves. We cannot apply a Type II move between $w_i$ and $w_j$ because they are connected.  If we apply a Type II move to $w_i$ and some other $w_k$ at least one of $w_i$ or $w_k$ is connected to $w_j$ (one can see this from Proposition \ref{prop: moreEdges}), so we still have a path from $w_i$ to $w_j$ because $w_i$ and $w_k$ are connected. An identical argument shows we still have a path from $w_i$ to $w_j$ if we swap $w_j$ and some $w_k$. Finally, if we apply a Type II move to two vertices neither of which are $w_i$ or $w_j$, then $w_i$ and $w_j$ are still clearly connected.

Now, if $w_i$ and $w_j$ are connected by some path $(w_i, w_{\alpha_1}, \dots, w_{\alpha_n}, w_j)$, iterating the above argument for each edge in the path shows $w_i$ and $w_j$ are still connected after any Type II move.
%
%
\end{proof}

\section{The counterexample} \label{sec: counterEx}

\noindent We are now ready to present the counterexample to the noetherianity of the poset $(\cM,\sqsubseteq)$. The idea behind the counterexample is as follows. We showed that two perfect matchings $\Gamma,\Gamma'$ are comparable, i.e. $\Gamma \sqsubset \Gamma'$, if and only if their corresponding permutations $p_\Gamma, p_{\Gamma'}$ are comparable, i.e. $p_\Gamma \leq_i p_{\Gamma'}$ (Proposition \ref{prop: matchingConnection}). 
%
%

So to prove that $(\cM,\sqsubseteq)$ is not noetherian, it suffices to produce an infinite chain of graphs $G_1,G_2,\dots$, prove that these graph are permutation graphs by labeling them, and argue that for any $G_i$ and $G_j$ with $i < j$, we cannot transform $G_i$ into $G_j$ by applying Type I and II moves to the underlying permutations. We use graphs, because it is much easier to work with what Type I and II moves induce on the graph theoretic side, than to work with the permutations themselves. Furthermore, we make this connection to graph theory because there are many well studied antichains on graphs and we will use one such antichain to produce an antichain in $(\cM,\sqsubseteq)$.

\begin{theorem}
The poset $(\cM,\sqsubseteq)$ is not noetherian.
\end{theorem}

We break this proof into many small pieces to make it easier to follow. We begin by presenting a chain of graphs. We argue these graphs are permutation graphs and associate a permutation to each one. Then we present each piece of the proof that these permutations are not comparable under $\leq_i$ as Lemmas and use this to show the chain we start with is an antichain. Throughout this proof, when we speak of applying a Type II move to two vertices in a graph, we mean applying the Type II move to the labels in the underlying permutation and tracking what this does to the permutation graph. We are always actually working with the permutations, but using the graphs to keep track of the inversions in the permutations.

\begin{proof}
A well known antichain for the subgraph order on graphs is the fork antichain $F_1,F_2,\dots$ where $F_k$ is the graph
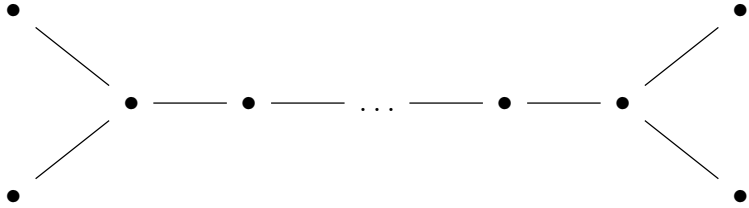
\[
\begin{tikzcd}
\bullet \ar[dash,dr]&&&&&&\bullet \ar[dash,dl]\\
&\bullet \ar[dash,r] \ar[dash,dl]&\bullet \ar[dash,r]&\dots\ar[dash,r]&\bullet \ar[dash,r]&\bullet \ar[dash,dr]&\\
\bullet &&&&&&\bullet
\end{tikzcd}
\]
i.e. the path on $k$ vertices with an additional two vertices connected to both the beginning and end of the path. We call the degree $3$ vertices at the beginning and end of the fork the left and right \defn{fork vertices} respectively, and the degree $1$ vertices the \defn{leaves of the fork}. We will also use this as an antichain for our poset. 

\begin{lemma} \label{lem: permAssoc}
Let $F_{2n}$ represent the fork on $2n+4$ vertices. Each $F_{2n}$ is a permutation graph, and to it we can associate a permutation $p_{2n}$.
\end{lemma}

\begin{proof}
We find $p_{2n}$ by labeling the vertices of $F_{2n}$. This both proves $F_{2n}$ is a permutation graph and associates a particular permutation to $F_{2n}$. 

Permutation graphs are characterized by \cite[Theorem 3.2]{permGraph}, so it suffices to show there is a labeling of the vertices in the fork that satisfies \prop{P1} and \prop{P2} as seen in Section \ref{sec: permGraphs}. One such labeling for $F_{2n}$ is given by: leaves on the left fork vertex labeled by $1$ and $2$, leaves on the right fork vertex labeled by $2n+3,2n+4$ and the path in between the fork vertices alternating with the pattern $4,3,6,5,\dots,2n+2,2n+1$ from left to right.

Indeed, this labeling satisfies \prop{P1} because there is no increasing path of length $3$ or greater in the graph. That is, for any edge $(i,j)$ with $i < j$ there is never an edge $(j,k)$ with $j < k$, so we trivially satisfy the transitivity property.

As for \prop{P2} one can first verify that the forks satisfy this property. The left leaves are always labeled with $1$ and $2$ and the left fork connected to them is labeled $4$. But the edges $(3,4)$ and $(2,4)$ are in the graph, so \prop{P2} is satisfied here. The exact same analysis shows \prop{P2} is satisfied by the right fork. As for the path connecting the fork vertices, for $2 \leq k \leq n+1$, there are edges of the form $(2k,2k-1)$ and for $2 \leq k \leq n$ there are edges of the form $(2k-1,2k+2)$. The first type of edge trivially satisfies \prop{P2}. For the second type of edge, notice that for $2 \leq k \leq n$, $2k-1$ is always connected to $2k$, and $2k+2$ is connected to $2k+1$. This shows \prop{P2} is also satisfied for these vertices. This covers all possible cases. 
\end{proof}

Throughout this proof, $F_{2n}$ will represent the fork graph on $2n+4$ vertices and $p_{2n}$ will represent the permutation associated to $F_{2n}$ in Lemma \ref{lem: permAssoc}. There is a similar labeling for $F_{2n+1}$, but the even forks suffice to produce an antichain. 

We will argue that we cannot transform $F_{2n}$ into any $F_{2m}$ using Type I or II moves on the corresponding permutations. We will often make use of the fact that Type I moves imply the induced subgraph relation and that Type II moves strictly increase the number of edges while maintaining the number of vertices.

To get from $F_{2n}$ to $F_{2m}$ we must perform $2(m-n)$ Type I moves, to add $2(m-n)$ vertices, since Type II moves do not add vertices. The number of Type II moves we are allowed to perform is bounded above by $2(m-n) - \beta$ where $\beta$ is the number of edges we gain from the Type I moves. This follows from Proposition \ref{prop: moreEdges} because Type II moves always increase the number of edges. We now need a result about which Type I moves add a vertex without adding any edges.

\begin{lemma}
The only Type I moves we can perform to $p_{2n}$ that do not also add an edge to the permutation graph are when we shift all the elements up by $\ell$ and add $12\cdots \ell$ to the beginning of the permutation, or add $(2n+5)(2n+6) \cdots (2n+k)$ to the end of the permutation. 
\end{lemma}

\begin{proof}
Since $F_{2n}$ is connected, if we try to add an element somewhere in the middle of the permutation, the only way this new vertex would have valence $0$ is if all the elements to the left of it in the permutation were less than it and all the elements to the right of it were larger than it. However, this implies the corresponding permutation graph is disconnected, which is not true for any $F_{2n}$.
\end{proof}

We call these new degree $0$ vertices we can add via Type I moves \defn{pivot vertices}. To summarize, the number of Type II moves we are allowed to perform is bounded above by the number of pivot vertices we add, and all pivot vertices are necessarily labeled by elements either strictly smaller or strictly larger than all the original elements from $F_{2n}$.
%
%
%
%

We cannot use Type I moves to transform $F_{2n}$ into $F_{2m}$ because $F_{2n}$ is not an induced subgraph of $F_{2m}$, indeed it is not even a subgraph of $F_{2m}$. So we must perform Type II moves at some point.
 
This implies we must add some pivot vertices to $F_{2n}$. But then we need to perform Type II moves to connect these vertices to the pre-existing graph. We can actually say something stronger, 
 
 \begin{lemma} \label{lem: involve0}
Whenever we apply a Type II move, it must involve two vertices that are not connected by any path.
\end{lemma}

\begin{proof}
If we performed a Type II move on two vertices which are connected by a path, this means they are both part of a connected subgraph. Consider the maximal connected subgraph they are a part of. Suppose this subgraph has $N$ vertices. Since it is connected, it has at least $N-1$ edges. A Type II move increases the number of edges in the subgraph. We therefore create a cycle in this subgraph. Proposition \ref{prop: preserveCycles} then implies any subsequent Type I or II moves will preserve this cycle, so the resulting graph could not be a tree, i.e. the resulting graph could not be any $F_{2m}$.
\end{proof}

%
%
%

An immediate corollary is that whenever we perform a Type II move, it must involve at least one pivot vertex because all other vertices are automatically part of $F_{2n}$ and therefore part of the same connected subgraph and there is no way to separate these vertices (Proposition \ref{prop: connected}).

%
%
%

The key observation, now, is that we cannot use Type I or Type II moves to remove or change the fork vertex at the beginning or end of the graph. By symmetry, it suffices to consider the right fork. These forks correspond to the sub-permutation $(2n+2)(2n+3)(2n+4)(2n+1)$. When we add pivot vertices, the end of the permutation becomes
 \[
 (2n+2)(2n+3)(2n+4)(2n+1)(2n+5)(2n+6)\cdots(2n+k).
 \]
It is clear that Type I moves do not change or remove the fork vertex. When we apply an allowable Type II move involving at least one of the pivot vertices, $2n+1$ will still always have three larger entries appearing before it. Indeed, if we tried to swap $(2n+1)$ with a pivot vertex added at the beginning of the permutation, that pivot vertex would become connected to every other vertex in $F_{2n}$. This clearly creates a cycle, which cannot occur by Proposition \ref{prop: preserveCycles}. If we swap $(2n+1)$ with a pivot vertex added at the end of the permutation, this only increases the number of larger entries appearing before $(2n+1)$. Any Type II move applied to a vertex other than $(2n+1)$ with a pivot vertex can only increase the number of larger entries appearing before $(2n+1)$, since the only way to move an entry that is larger than $(2n+1)$ to its right is to replace it with an even larger entry.

This means after any application of allowable Type I and II moves, the image of $(2n+1)$ will have valence at least $3$. A similar argument shows that the image of $4$ will also have valence $3$. Only two vertices in any fork graph have this property, the fork vertices. This implies that using Type I and II moves, we must always send the fork vertices to fork vertices. As a result, the only way to send $F_{2n}$ to $F_{2m}$ is to extend the path between the two fork vertices. 

To do this, we must perform a Type II move on one of the vertices in the path between the two fork vertices and a pivot vertex. Indeed, we cannot accomplish this with just Type I moves because $F_{2n}$ is not an induced subgraph of $F_{2m}$. This implies we must use both Type I and II moves.  We showed that when we apply a Type II move, it must involve at least one pivot vertex in Lemma \ref{lem: involve0}. If the Type II move did not involve a vertex from the path between the two fork vertices, the path of length $2n$ between the two pivot vertices in $F_{2n}$ would remain, but the minimal path between the two fork vertices in any $F_{2m}$ is length $2m > 2n$. However,

\begin{lemma} \label{lem: pathFixed}
If we try to apply a Type II move swapping any pivot vertex with any vertex in the path between the fork vertices, this will create a cycle.
\end{lemma}

\begin{proof}
If we try to swap a pivot vertex $\alpha$ added to the end of the permutation with a vertex in between the fork vertices, we have the subpermutation $\alpha (2n+3)(2n+4)(2n+1)$ with $\alpha$ larger than all the vertices from $F_{2n}$, so in particular $\alpha > 2n+4$. This contains the cycle $\alpha (2n+1)(2n+3)\alpha$. A similar argument works for a pivot vertex added to the beginning of the permutation.
\end{proof}

As a result, we cannot perform a Type II move between any pivot vertex and a vertex between the fork vertices by Proposition \ref{prop: preserveCycles}. As we already mentioned, in any fork graph $F_{2m}$, the minimal path between the two fork vertices is length $2m$. Lemma \ref{lem: pathFixed} shows we can never lengthen the path between two fork vertices using a combination of Type I and II moves, and we also cannot change the fork vertices using Type I or II moves. As a result, we can never transform the permutation corresponding to $F_{2n}$ into the permutation corresponding to $F_{2m}$. Stated another way, this chain of permutations is indeed an antichain.
%
\end{proof}

We follow this proof with many examples to illustrate the phenomenon appearing in the proof and to explore the counterexample itself.

\begin{example}
We show a few of the labeled even forks, $F_2$ is
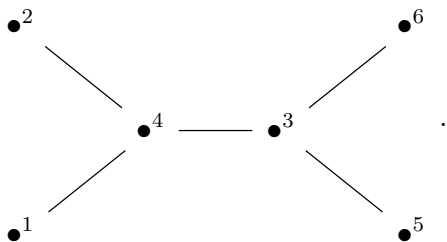
\[
\begin{tikzcd}
\bullet^{2} \ar[dash,dr]&&&\bullet^{6} \ar[dash,dl]\\
&\bullet^{4} \ar[dash,r] \ar[dash,dl]&\bullet^{3} \ar[dash,dr]&\\
\bullet^{1} &&&\bullet^{5}
\end{tikzcd}.
\]
This is the permutation graph for $412563$. This then corresponds to the perfect matching
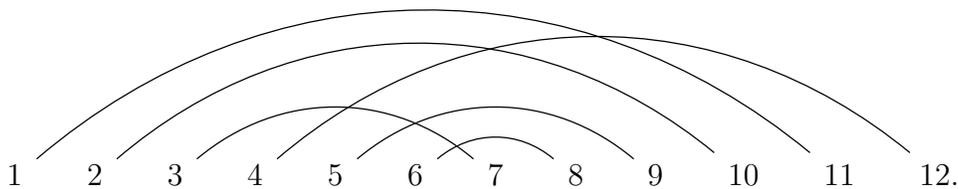
\[
\begin{tikzcd}[sep = small]
1\ar[bend left = 40,dash,rrrrrrrrrr]&2\ar[bend left = 40,dash,rrrrrrrr]&3\ar[bend left = 40, dash,rrrr]&4 \ar[bend left = 40,dash,rrrrrrrr]&5\ar[bend left = 40, dash, rrrr]&6 \ar[bend left = 40, dash, rr]&7&8&9&10&11&12.
\end{tikzcd}
\]
The next even fork, $F_4$, is
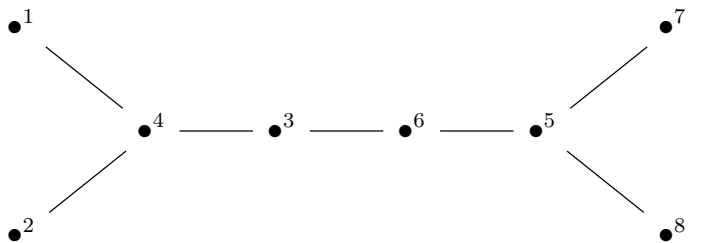
\[
\begin{tikzcd}
\bullet^{1} \ar[dash,dr]&&&&&\bullet^{7} \ar[dash,dl]\\
&\bullet^{4} \ar[dash,r] \ar[dash,dl]&\bullet^{3} \ar[dash,r] &\bullet^{6} \ar[dash,r]&\bullet^{5} \ar[dash,dr]&\\
\bullet^{2} &&&&&\bullet^{8}
\end{tikzcd},
\]
which is the permutation graph for $41263785$, which corresponds to the intertwined perfect matching,
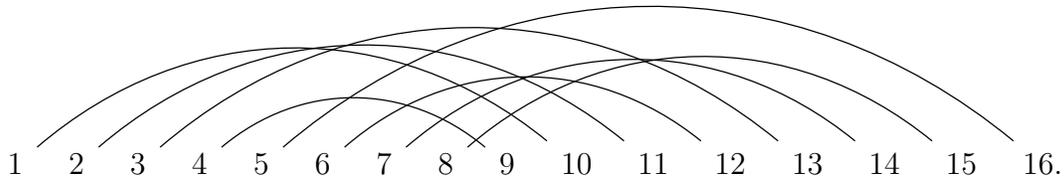
\[
\begin{tikzcd}[sep = tiny]
1\ar[bend left = 40,dash,rrrrrrrrr]&2\ar[bend left = 40,dash,rrrrrrrrr]&3\ar[bend left = 40, dash,rrrrrrrrrr]&4 \ar[bend left = 40,dash,rrrrr]&5\ar[bend left = 40, dash, rrrrrrrrrrr]&6 \ar[bend left = 40, dash, rrrrrr]&7\ar[bend left = 40, dash, rrrrrrr]&8\ar[bend left = 40, dash, rrrrrrr]&9&10&11&12&13&14&15&16.
\end{tikzcd}
\]
As one can see, trying to just work with these perfect matchings is rather difficult. It is not clear how one might argue that no sequence of moves could transform the previous diagram into this one, but this is the case.
\end{example}

\begin{example}
Now suppose we wish to perform a Type I move to add a degree $0$ vertex to $F_2$. $F_2$ corresponds to the permutation $412563$ from our labeling. If we try to add any entry between the existing entries, we necessarily also add edges to the graph. For example suppose we perform the Type I move $412563 \to 4125673$ the resulting graph is
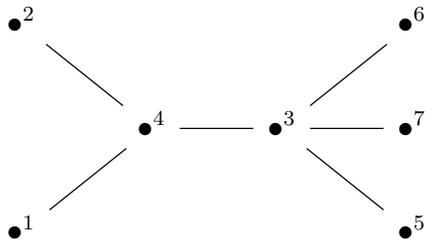
\[
\begin{tikzcd}
\bullet^{2} \ar[dash,dr]&&&\bullet^{6} \ar[dash,dl]\\
&\bullet^{4} \ar[dash,r] \ar[dash,dl]&\bullet^{3} \ar[dash,dr]\ar[dash,r]&\bullet^7\\
\bullet^{1} &&&\bullet^{5}
\end{tikzcd}.
\]
\end{example}

\begin{example}
If we again consider $F_2$, we will now explore the content of Lemma \ref{lem: involve0}. If we attempt to apply any Type II move to $412563$, Lemma \ref{lem: involve0} implies we will create a cycle. Indeed, suppose we try to swap $2$ and $3$. We then end up with the permutation $413562$ which corresponds to the graph 
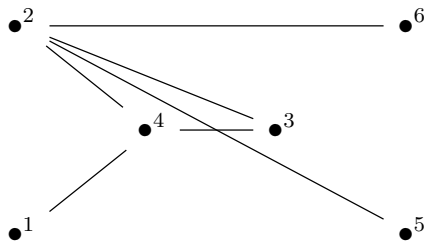
\[
\begin{tikzcd}
\bullet^{2} \ar[dash,dr]&&&\bullet^{6} \ar[dash,lll]\\
&\bullet^{4} \ar[dash,r] \ar[dash,dl]&\bullet^{3} \ar[dash,ull]&\\
\bullet^{1} &&&\bullet^{5} \ar[dash,uulll]
\end{tikzcd}.
\]
This clearly contains the cycle $2432$.

Now let us consider Lemma \ref{lem: pathFixed}. We will add a pivot vertex to $F_4$, say we do this and get the permutation $412637859$. If we try to apply a Type II move with $9$ and any vertex on the path between the forks we necessarily get a cycle. Suppose we tried to use a Type II move to swap $9$ and $6$. We end up with the permutation $412937856$ which corresponds to the graph
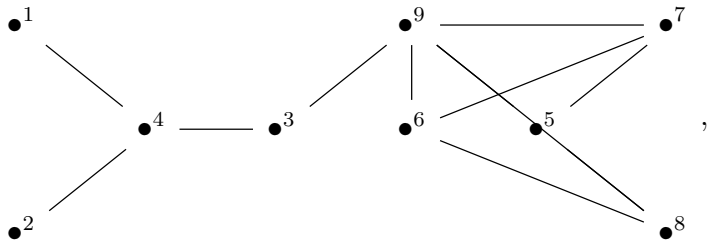
\[
\begin{tikzcd}
\bullet^{1} \ar[dash,dr]&&&\bullet^9\ar[dash,dl]\ar[dash,d] \ar[dash,dr]\ar[dash,rr]\ar[dash,ddrr]&&\bullet^{7} \ar[dash,dl]\\
&\bullet^{4} \ar[dash,r] \ar[dash,dl]&\bullet^{3} &\bullet^{6} \ar[dash,urr]\ar[dash,drr]&\bullet^{5} \ar[dash,dr]&\\
\bullet^{2} &&&&&\bullet^{8}
\end{tikzcd},
\]
which has many cycles. Indeed, the only Type II moves that do not create a cycle involving $9$ are to swap it with $5$ or $8$ which respectively correspond to the graphs,
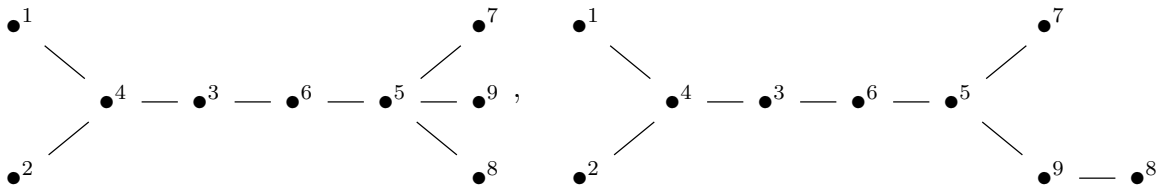
\[
\begin{tikzcd}[sep=small]
\bullet^{1} \ar[dash,dr]&&&&&\bullet^{7} \ar[dash,dl]\\
&\bullet^{4} \ar[dash,r] \ar[dash,dl]&\bullet^{3} \ar[dash,r] &\bullet^{6} \ar[dash,r]&\bullet^{5} \ar[dash,dr]\ar[dash,r]&\bullet^9\\
\bullet^{2} &&&&&\bullet^{8}
\end{tikzcd},
\quad
\begin{tikzcd}[sep=small]
\bullet^{1} \ar[dash,dr]&&&&&\bullet^{7} \ar[dash,dl]&\\
&\bullet^{4} \ar[dash,r] \ar[dash,dl]&\bullet^{3} \ar[dash,r] &\bullet^{6} \ar[dash,r]&\bullet^{5} \ar[dash,dr]&&\\
\bullet^{2} &&&&&\bullet^{9} \ar[dash,r]&\bullet^8
\end{tikzcd}
\]
This is the overarching idea behind the proof. We must end up with a connected tree, but because Type I and II moves preserve cycles, and Type II moves always add edges, many Type II moves on a tree would create a cycle, which severely limits when we can use them.
\end{example}


\section{Going forward} \label{sec: goingForward}
The proof that this chain of graphs yields a counterexample relies heavily on the fact that we do not have a move that maintains both the number of vertices and the number of edges. Type I moves always add vertices, and potentially edges. They are also very rigid in that they preserve the order of the original permutation. Type II moves always add more edges, but do not add any vertices. This suggests that additional moves are necessary to make $(\cM,\sqsubseteq)$ noetherian. In particular, one needs to add moves that do not add edges or vertices, but merely swap edges around. Such moves have the potential to break Proposition \ref{prop: preserveCycles} and therefore potentially break the counterexample.

For example, we believe the move $2341 \to 4123$ preserves initial ideals, which corresponds to
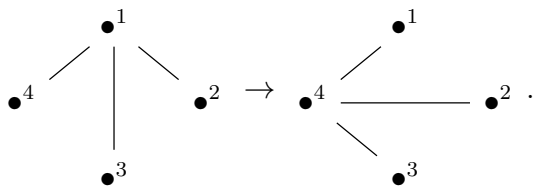
\[
\begin{tikzcd}[sep=small]
&\bullet^1\ar[dash,dd]\ar[dash,dr]\ar[dash,dl]&\\
\bullet^4&&\bullet^2\\
&\bullet^3&
\end{tikzcd}
\to
\begin{tikzcd}[sep=small]
&\bullet^1\ar[dash,dl]&\\
\bullet^4&&\bullet^2\ar[dash,ll]\\
&\bullet^3\ar[dash,ul]&
\end{tikzcd}.
\]
Another move that we believe preserves initial ideals is $231 \to 312$. This is similar to the previous new move in that it relates two permutations with the same permutation graph. This move would immediately break the counterexample because it would allow us to send $412563$, which is the permutation corresponding to $F_2$, to $412635$ which has permutation graph
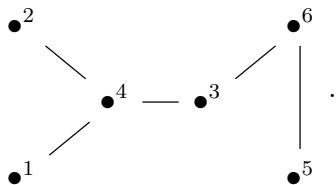
\[
\begin{tikzcd}[sep=small]
\bullet^{2} \ar[dash,dr]&&&\bullet^{6} \ar[dash,dl]\ar[dash,dd]\\
&\bullet^{4} \ar[dash,r] \ar[dash,dl]&\bullet^{3}&\\
\bullet^{1} &&&\bullet^{5}
\end{tikzcd}.
\]
This graph is easily seen to be an induced subgraph of $F_4$. 

\subsection{Equivariant Initial Ideals}
Initial ideals have played an important role in classical commutative algebra. One can often derive many important properties of ideals and algebras from their initial counterparts. One key example of this is determinantal ideals, see \cite{detIdeals} for a nice survey. 

Recently, researchers have been investigating how classical areas of commutative algebra behave in an equivarant setting (often the equivariant analogues behave differently). For example, Snowden investigated $\GL$-prime ideals, i.e. prime ideals in tca's, and discovered an effective method for analyzing them \cite{tcaprimes}. The author and Snowden then expanded this to describe an effective method for analyzing equivariant prime ideals for infinite dimensional supergroups \cite{SuperPrimes}. Bik, Draisma, Eggermont and Snowden are also currently investigating $\GL$-varieties \cite{GLVarieties}.

Sam and Snowden laid the foundations for an equivariant Gr\"obner theory in \cite{grob}, but as we have seen these methods will need to be expanded to apply more generally. Taking cues from classical commutative algebra, if one wanted to develop a robust equivariant Gr\"obner theory, it would also be important to understand equivariant initial ideals. Indeed, one way to classify all possible moves is to understand the structure of initial ideals in $\Sym(\Sym^2(\bC^\infty))$, and in tca's more generally. Each move is a partial picture of the initial ideal structure. 

We are currently investigating exactly this for the tca $\Sym(\Sym^2(\bC^\infty))$. We now outline some other potential avenues for future work stemming from this paper.

\subsection{Are intertwined matchings enough?} The noetherianity of the subposet of intertwined perfect matchings is an easier problem to approach than the noetherianity of $\cM$. We actually believe the noetherianity of this subposet is a good indicator for the noetherianity of the original poset. In particular, we are investigating the following question,

\begin{question}
If the class of intertwined perfect matchings under some extension of $\sqsubseteq$ is well-quasi-ordered, then is $\cM$ also well-quasi-ordered?
\end{question}

When we say some extension, we mean adding additional types of moves. We now sketch the idea behind this question. Notice, we can separate any perfect matching into a collection of intertwined perfect matchings. One systematic way to do this is to read from left to right in the perfect matching, coloring the edges with color $c_1$ until we reach a terminal vertex. We then label the edge connected to the next origin vertex by $c_2$ and continue to do this until we reach a terminal vertex labeled by $c_2$, then repeat this process with $c_3$ etc. To see this process at work consider the graph
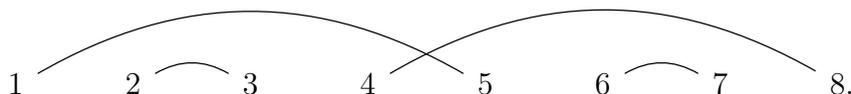
\[
\begin{tikzcd}
1\ar[dash, bend left = 30, rrrr]&2 \ar[dash, bend left = 30, r] &3 &4\ar[dash, bend left = 30,rrrr]&5&6\ar[dash,bend left = 30,r]&7&8.
\end{tikzcd}
\]
This breaks into the two intertwined perfect matchings colored as follows
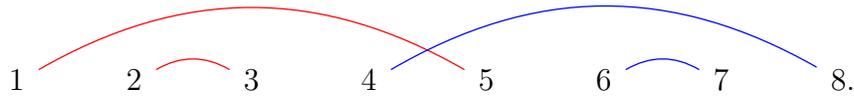
\[
\begin{tikzcd}
1\ar[dash, bend left = 30, rrrr,color=red]&2 \ar[dash, bend left = 30, r,color=red] &3 &4\ar[dash, bend left = 30,rrrr,color=blue]&5&6\ar[dash,bend left = 30,r,color=blue]&7&8.
\end{tikzcd}
\]
The introduction of Type II moves implies that in any potential infinite antichain of perfect matchings in $(\cM, \sqsubseteq)$, there must be a bounded number of intertwined pieces in any of the perfect matchings that appear. Indeed, if this were not the case, it implies that for any $n \gg 0$, there is a perfect matching with a submatching of the form,
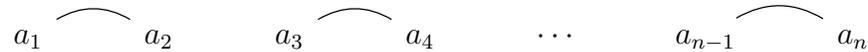
\[
\begin{tikzcd}
a_1\ar[dash, bend left = 30, r] &a_2 & a_3 \ar[dash, bend left = 30, r] & a_4 & \cdots &a_{n-1} \ar[dash, bend left = 30, r] & a_{n}
\end{tikzcd}
\]
with $a_1 < a_2 < \cdots < a_n$. If the first matching in the antichain is on $k$ vertices, we can use Type II moves to transform this matching into the matching
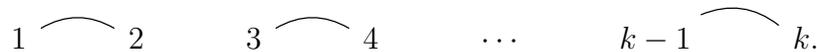
\[
\begin{tikzcd}
1\ar[dash, bend left = 30, r] &2 & 3 \ar[dash, bend left = 30, r] & 4 & \cdots &{k-1} \ar[dash, bend left = 30, r] & k.
\end{tikzcd}
\]
We can then use Type I moves to turn this matching into the previous one. Filling in all the other edges with more Type I moves shows that these two elements are comparable. As a result, in any infinite antichain, there must be a bound on the number of intertwined pieces. Equivalently, we can color the edges of the perfect matchings in any infinite antichain using a fixed bounded number of colors so that each color is an intertwined perfect matching. The idea, then, is to prove that noetherianity holds for each intertwined piece and argue that one can glue these pieces together to get noetherianity in general.

This is also an abstract way to argue that the antichain demonstrating that Type I moves are not enough to get noetherianity in \cite[Example 5.1]{sym2noeth} is a good chain when one introduces Type II moves. Indeed, it is not hard to see that there is no bound on the number of intertwined pieces, therefore we eventually have a comparison with the introduction of Type II moves.

The current issue with this approach is that one cannot completely break apart the intertwined pieces of a matching, for example in the graph we considered earlier,
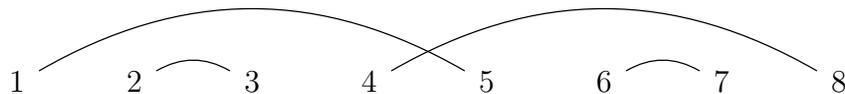
\[
\begin{tikzcd}
1\ar[dash, bend left = 30, rrrr]&2 \ar[dash, bend left = 30, r] &3 &4\ar[dash, bend left = 30,rrrr]&5&6\ar[dash,bend left = 30,r]&7&8
\end{tikzcd}
\]
when trying to break up these diagrams into pieces, the first disjoint piece would be the subgraph on the vertices $\{1,2,3,5\}$ and the second would be the subgraph on the vertices $\{4,6,7,8\}$.  We could not use a Type II move on the vertices $\{4,6,7,8\}$ because $5$ is connected to a vertex before $6$. Indeed, if we try to compare this perfect matching to the following perfect matching,
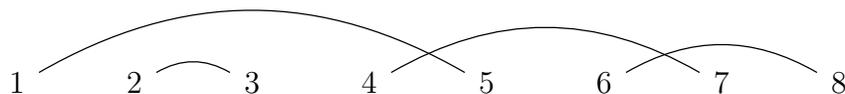
\[
\begin{tikzcd}
1\ar[dash, bend left = 30, rrrr]&2 \ar[dash, bend left = 30, r] &3 &4\ar[dash, bend left = 30,rrr]&5&6\ar[dash,bend left = 30,rr]&7&8
\end{tikzcd}
\]
both the intertwined pieces are comparable, but we cannot apply the Type II(a) move to the vertices $\{4,6,7,8\}$ because $5$ is connected to $1$. So these graphs are not actually comparable with our current moves.

So, although one can break up all the diagrams into a finite number of intertwined perfect matchings, it is possible that sometimes we could not use Type II moves. There are ways one could approach this issue, in particular by adding moves that allow one to work with the intertwined pieces separately. We will not include more details at the moment because we are only trying to motivate the importance of the subclass of intertwined perfect matchings.

At the very least, when one introduces new types of moves it should be easier to test whether the subposet of intertwined perfect matchings becomes noetherian. This could then be a good indicator that these moves are enough to make the poset $(\cM,\sqsubseteq)$ noetherian.

\subsection{Permutation Perspective.}
The subposet of intertwined perfect matchings is order isomorphic as a poset to the poset of permutations with the partial order induced by $\sqsubseteq$ (Proposition \ref{prop: matchingConnection}), so if one could prove that the corresponding class of permutations is well-quasi-ordered under allowable moves, this would imply the intertwined perfect matchings were well-quasi-ordered as well.  

Indeed, a byproduct of adding more moves seems to be forbidding certain patterns. For example, Type II moves forbid the permutation $\ell (\ell-1) \cdots 2 1$ from occurring in any element of an antichain that begins with a permutation of length $\ell$ because we can turn any permutation of length $\ell$ into this one, then use Type I moves to embed to this subpermutation. So one approach to proving noetherianity of at least the intertwined perfect matchings is introducing enough moves to forbid enough permutations so that the allowable permutations fall into a class that is known to be well-quasi-ordered.
%

Over the course of many years, researchers have developed various techniques for proving permutation classes are well-quasi-ordered. See the following papers for reference \cite{wqoPermGraph, ProfClass, GeomGridClass, GridClassFib, SmallPermClass, InflationsOfGeom, SimplePerms}.
We will not elaborate further, we merely point this out and include references because there is a rich and ongoing theory concerned with proving classes of permutations forbidding certain patterns are well-quasi-ordered. This paper suggests there is a connection between the noetherianity of tcas and this branch of combinatorics which is worth exploring further.

\subsection{Bruhat Order}

The Bruhat order is an extremely well studied partial order on permutations of the same size. Likewise, order isomorphism is a well studied partial order on permutations of any size. The Bruhat order is clearly well-quasi-ordered because the poset is finite. As we have discussed, the poset of all permutations with order isomorphism as the partial order is not well-quasi-ordered. One could ask if adding the Bruhat order is enough to make the set of all permutations well-quasi-ordered.

More explicitly, for two permutation $\sigma,\tau$ we say that $\sigma \preceq' \tau$ if either $\sigma \leq_{B} \tau$ or $\sigma \leq_o \tau$, where the first order is the Bruhat order and the second is an order isomorphism. We then say $\sigma \preceq \tau$ is there is some sequence of permutations $\sigma \preceq' \rho_1 \preceq' \rho_2 \preceq' \cdots \preceq' \rho_n \preceq' \tau$ relating $\sigma$ and $\tau$. The question is then whether $(\cP,\preceq)$ is well-quasi-ordered, where $\cP$ is the set of all permutations.

This is very similar to the poset that arises from studying $\Sym(\Sym^2(\bC^\infty))$. Indeed, the same proof technique as above can show that this partial order is also not noetherian. So introducing the Bruhat order is actually not enough to get noetherianity. The key property is Proposition \ref{prop: preserveCycles}, which still holds for this partial order. We will not prove this, but thought it worth mentioning because both of these partial orders have received a lot of attention. We think it is a worthwhile question, now closely tied to the noetherianity of $\Sym(\Sym^2(\bC^\infty))$, to ask what possible moves or relations one can introduce on permutations of the same size so that the order isomorphism partial order becomes noetherian.

Up to this point, researchers have been focused on finding subclasses of the poset of all permutations that are well-quasi-ordered under order isomorphism, but this heads in the other direction. Rather than shrinking the size of the set under consideration, we are asking what additional and natural moves one could add to expand the partial order to make the whole poset noetherian.


\begin{thebibliography}{ABKLV}

\bibitem[AA]{SimplePerms} Michael H. Albert, Mike D. Atkinson. Simple permutations and pattern restricted permutations. {\em Discrete Mathematics}, {\bf 300} (2005), no. ~1, 1--15.

\bibitem[AABRV]{GeomGridClass} Michael H. Albert, Mike D. Atkinson, Mathilde Bouvel, Nik Ru{\v{s}}kuc, Vincent Vatter. Geometric grid classes of permutations. {\em Trans. of the Am. Math. Soc.}, {\bf 365} (2013), no. ~11, 5859--5881.

\bibitem[ABKLV]{wqoPermGraph} Aistis Atminas, Robert Brignall, Nicholas Korpelainen, Vadim Lozin, Vincent Vatter. Well-Quasi-Order for permutation graphs omitting a path and a clique. {\it Electronic J. Combinatorics} {\bf 27} (2015), no. ~2.

\bibitem[ARV]{InflationsOfGeom} Michael H. Albert, Nik Ru{\v{s}}kuc, Vincent Vatter. Inflations of geometric grid classes of permutations. {\em Israel J. Math.} {\bf 205} (2015), no. ~1, 73--108.

\bibitem[BDES]{GLVarieties} Arthur Bik, Jan Draisma, Robert H. Eggermont, Andrew Snowden. {\em The geometry of polynomial representations}, in preparation.

\bibitem[BC]{detIdeals} Winfried Bruns, Aldo Conca. Gr\"obner bases and determinantal ideals. In \emph{Commutative algebra, singularities and computer algebra} (2003), Springer, Dordrecht, 9--66.

\bibitem[KR]{permGraph} Youngmee Koh, Sangwook Ree. Determination of permutation graphs. {\it Honam Mathematical J.} {\bf 27} (2005), no.~2, 183--194.

%
%
%
\bibitem[CEF]{fimodule} Thomas Church, Jordan S. Ellenberg, Benson Farb. FI-modules and stability for representations of symmetric groups. {\it Duke Math. J.} {\bf 164} (2015), no.~9, 1833--1910. \arxiv{1204.4533v4}
%
\bibitem[Dr]{draisma} Jan Draisma. Topological Noetherianity of polynomial functors. \emph{J.\ Amer.\ Math.\ Soc.}\ {\bf 32}(3) (2019), pp.\ 691--707. \arxiv{1705.01419}
%
%
%
%
\bibitem[HV]{GridClassFib} Sophie Huczynska, Vincent Vatter. Grid classes and the Fibonacci dichotomy for restricted permutations. {\em Electronic J. Combin.}. {\bf 13} (2006).

\bibitem[La1]{heckeCat} Robert P. Laudone. Representation Stability for sequences of $0$-Hecke modules. \arxiv{1910.07036}

\bibitem[La2]{Plucker} Robert P. Laudone. Syzygies of secant ideals of Pl\"ucker-embedded Grassmannians are generated in bounded degree. \arxiv{1803.04259}

\bibitem[LS]{SuperPrimes} Robert P. Laudone, Andrew Snowden. Equivariant primes ideals for infinite dimensional supergroups. \arxiv{2103.03152}

\bibitem[Lav]{wqoFinSeq} Richard Laver. Well-quasi-orderings and sets of finite sequences. {\it Math. Proc. Camb. Philos. Soc.} {\bf 79} (1976), 1--10.

\bibitem[MV]{ProfClass} Maximillian M. Murphy, Vincent R. Vatter. Profile classes and partial well-order for permutations. {\em Electronic J. Combinatorics} {\bf 9} (2003), no. ~2.

\bibitem[NSS1]{sym2noeth} Rohit Nagpal, Steven V Sam, Andrew Snowden. Noetherianity of some degree two twisted commutative algebras. {\it Selecta Math. (N.S.)} {\bf 22} (2016), no.~2, 913--937. \arxiv{1501.06925v2}

\bibitem[NSS2]{periplectic} Rohit Nagpal, Steven V Sam, Andrew Snowden. Noetherianity of some degree two twisted skew-commutative algebras. {\it Selecta Math. (N.S.)} {\bf 25} (2019), no.~1. \arxiv{1610.01078v2}
%
%
\bibitem[Pra]{compPerm} Vaughan R. Pratt. Computing permutations with double-edged queues, parallel stacks and parallel queues. {\it Proc. ACM Symp. Theory of Computing} {\bf 5} (1973), 268--277.

\bibitem[Sa]{Veronese} Steve V Sam. Ideals of bounded rank symmetric tensors are generated in bounded degree. {\em Inv. mathematicae} {\bf 207} (2017), no. ~1, 1--21.

%
\bibitem[SS1]{grob} Steven~V Sam, Andrew Snowden. Gr\"obner methods for representations of combinatorial categories. {\it J. Amer. Math. Soc.} {\bf 30} (2017), 159--203.

\bibitem[SS2]{introtca} Steven~V Sam, Andrew Snowden. Introduction to twisted commutative algebras. Preprint 2012. \arxiv{1209.5122}

\bibitem[SS3]{spnoeth} Steven V~Sam, Andrew Snowden. Sp-equivariant modules over polynomial rings in infinitely many variables. \arxiv{2002.03243}
%

\bibitem[Sn]{tcaprimes} Andrew Snowden. The spectrum of a twisted commutative algebra. \arxiv{2002.01152}
\bibitem[SB]{infAntiPerm} Daniel A. Spielman, Mikl\'os B\'ona. An Infinite Antichain of Permutations, {\em Electronic J. Combinatorics} {\bf 7} (2000), no. ~2.
%
%
%
\bibitem[Va]{SmallPermClass} Vincent Vatter. Small permutation classes. {\em Proc. Long. Math. Soc.} {\bf 103} (2011), no. ~5, 879--921.
  
\end{thebibliography}
\end{document}